\documentclass[11pt,a4paper]{article}

\usepackage[T1]{fontenc}
\usepackage[utf8]{inputenc}
\usepackage{newtxtext,newtxmath}
\usepackage[a4paper,left=2.2cm,right=2.2cm,top=2.0cm,bottom=2.0cm]{geometry}
\usepackage{microtype}
\usepackage{mathtools}
\usepackage{enumitem}
\usepackage[hidelinks]{hyperref}
\usepackage[
  backend=biber,
  style=authoryear-comp,
  maxcitenames=2,
  maxbibnames=99,
  uniquename=false,
  giveninits=true,
  doi=true,
  url=false,
  isbn=false
]{biblatex}
\usepackage{amsthm}
\theoremstyle{definition}
\newtheorem{theorem}{Theorem}
\newtheorem{corollary}[theorem]{Corollary}
\newtheorem{definition}{Definition}

\newcommand{\E}{\mathbb E}
\newcommand{\Pp}{\mathbb P}
\newcommand{\KL}{\mathrm{KL}}
\newcommand{\kl}{\mathrm{kl}}
\newcommand{\1}{\mathbf 1}
\newcommand{\X}{\mathcal X}
\newcommand{\Sstate}{\mathcal S}
\newcommand{\A}{\mathcal A}
\newcommand{\ThetaSet}{\Theta}

\title{\textbf{Statistical reduction before the target is known: two boundary results}}
\author{Rianne de Heide\\
\small Department of Applied Mathematics, University of Twente, Enschede, The Netherlands\\
\small Machine Learning Group, Centrum Wiskunde \& Informatica (CWI), Amsterdam, The Netherlands\\
\small \texttt{r.deheide@utwente.nl}}
\date{4 September 2026}

\begin{document}
\maketitle

\begin{abstract}
Suppose that the eventual use of data is not known when the data are reduced or collected. This note considers two simple boundary cases. In a finite statistical experiment, a statistic preserves the Bayes risk for every finite later decision problem if and only if it is sufficient. Hence, when the minimal sufficient statistic is one-to-one, exact preservation of all later decision problems permits no nontrivial reduction. We then consider adaptive sampling from $m$ independent Gaussian streams when an external query specifies the coordinate to be classified only after sampling stops. Under coordinatewise error control, the optimal symmetric average sample size is exactly $m$ times the one-coordinate optimum. A change-of-measure argument gives the corresponding pointwise lower bound in terms of binary relative entropy.
\end{abstract}

\section{Introduction}

When the inferential task is known in advance, reduction can be tailored to that task. If the task is specified only later, one can instead ask what has to be retained from data already observed, or how much has to be collected before the eventual target is known. This note considers one extreme case of each question.

For the first question, suppose that a reduced statistic must preserve the Bayes risk for every finite later decision problem. In a finite statistical experiment, this is exactly ordinary sufficiency. This is a finite specialization of the classical Blackwell--Bahadur theory: Blackwell's ordering compares experiments through their attainable risks, and equality corresponds to equivalence up to randomization \parencite{blackwell1951comparison,blackwell1953equivalent,torgersen1991comparison}; Bahadur developed the corresponding relation with sufficient statistics, including sequential experiments \parencite{bahadur1954sufficiency}. We give a direct finite-dimensional proof. A simple corollary gives a no-compression case: if distinct observations have nonproportional likelihood vectors, any statistic that preserves every later finite decision problem must be one-to-one on the support. Approximate preservation is related to Le Cam's deficiency \parencite{lecam1964sufficiency,torgersen1991comparison}, and recent work studies transformations that preserve minimum risk for every loss \parencite{gyorfi2023lossless}. We work with finite sample and parameter spaces throughout; in more general measurable spaces, Bayesian and classical notions of sufficiency can differ \parencite{blackwellramamoorthi1982bayes}.

For the second question, consider $m$ independent binary Gaussian experiments, one for each possible later target. The statistician may choose adaptively which coordinate to sample, but an external query names the coordinate whose sign is required only after sampling has stopped. Under a symmetric average-sample-size criterion, the optimum is exactly $m$ times the one-coordinate optimum. The reason is simple: each coordinate of any $m$-coordinate procedure can be emulated as a valid one-coordinate sequential test, while observations from the other coordinates can be generated as parameter-independent randomization.

Classical sequential testing supplies the one-coordinate benchmark \parencite{wald1947sequential,waldwolfowitz1948optimum}. Adaptive selection of experiments goes back at least to \textcite{chernoff1959sequential}, and controlled-sensing formulations allow the observation mechanism itself to be selected adaptively \parencite{nitinawarat2013controlled}. The change-of-measure argument used below is standard in sequential testing and fixed-confidence bandit theory \parencite{kaufmann2016complexity}. Sequential multiple-testing problems for independent streams usually impose stronger joint classification criteria; see, for example, \textcite{songfellouris2019sequential}. The theorem below instead uses coordinatewise error control for an external query; it does not cover a target selected from the stopped data.

\section{Bayes-risk preservation and sufficiency}

Let $\ThetaSet$ and $\X$ be finite sets and let
\[
    \mathcal E_X = \{P_\theta : \theta\in\ThetaSet\}
\]
be a statistical experiment on $\X$. Let
\[
    T:\X\to\Sstate
\]
be a statistic, where $\Sstate$ is finite, and write $Q_\theta$ for the distribution of $T(X)$ under $P_\theta$. Fix a prior $\pi$ satisfying $\pi(\theta)>0$ for every
$\theta\in\ThetaSet$, and write
\[
\bar P=\sum_{\theta\in\ThetaSet}\pi(\theta)P_\theta,
\qquad
\bar Q=\sum_{\theta\in\ThetaSet}\pi(\theta)Q_\theta.
\]

A finite decision problem consists of a finite action set $\A$ and a loss function
\[
    L:\ThetaSet\times\A\to[0,\infty).
\]
A randomized decision rule based on $X$ is a Markov kernel $\delta(a\mid x)$ from $\X$ to $\A$; a rule based on $T$ is a Markov kernel $\gamma(a\mid s)$ from $\Sstate$ to $\A$. Define the corresponding Bayes risks by
\begin{align*}
 R_X(\pi,L)
   &= \inf_\delta \sum_{\theta,x,a}
      \pi(\theta)P_\theta(x)\delta(a\mid x)L(\theta,a),\\
 R_T(\pi,L)
   &= \inf_\gamma \sum_{\theta,s,a}
      \pi(\theta)Q_\theta(s)\gamma(a\mid s)L(\theta,a).
\end{align*}
Since $T$ is a deterministic function of $X$, $R_X(\pi,L)\le R_T(\pi,L)$ for every decision problem.

In this finite setting, $T$ is sufficient for $\{P_\theta:\theta\in\ThetaSet\}$ if the conditional distribution of $X$ given $T=s$ can be chosen independently of $\theta$ for every $s$ having positive probability under at least one $P_\theta$. This is equivalent to the usual factorization and randomization characterizations of sufficiency in finite dominated experiments \parencite{bahadur1954sufficiency,torgersen1991comparison}.

The next definition is the equality case of Blackwell's decision criterion, specialized to the comparison between an experiment and a statistic generated from it \parencite{blackwell1951comparison,blackwell1953equivalent,torgersen1991comparison}.

\begin{definition}[Bayes-risk preserving statistic]
The statistic $T$ is \emph{Bayes-risk preserving at $\pi$} if
\[
      R_X(\pi,L)=R_T(\pi,L)
\]
for every finite action set $\A$ and every loss function $L$.
\end{definition}

Restricting attention to one full-support prior is enough. For any other prior $\rho$ on $\ThetaSet$, replacing $L(\theta,a)$ with $\rho(\theta)L(\theta,a)/\pi(\theta)$ turns its Bayes risk under $\rho$ into the Bayes risk under $\pi$.

\begin{theorem}[Equivalence with sufficiency]\label{thm:universal}
For the finite experiment above and a statistic $T$, the following statements are equivalent.
\begin{enumerate}[label=\textnormal{(\roman*)},leftmargin=2.2em]
\item $T$ is Bayes-risk preserving at $\pi$.
\item There exists a Markov kernel $K(x\mid s)$ from $\Sstate$ to $\X$, independent of $\theta$, such that
\[
       P_\theta(x)=\sum_{s\in\Sstate} Q_\theta(s)K(x\mid s)
       \qquad
       \text{for all }\theta\in\ThetaSet,\ x\in\X.
\]
\item $T$ is sufficient for $\{P_\theta:\theta\in\ThetaSet\}$.
\end{enumerate}
Thus equality of Bayes risk for all finite downstream decision problems is equivalent to ordinary sufficiency in this setting.
\end{theorem}

\begin{proof}
Assume (ii). For a randomized decision rule $\delta(a\mid x)$ based on $X$, define
\[
   \gamma(a\mid s)=\sum_{x\in\X}K(x\mid s)\delta(a\mid x).
\]
For every $\theta$, the action distribution obtained by applying $\gamma$ to $T(X)$ equals the action distribution obtained by applying $\delta$ to $X$. The two rules therefore have the same Bayes expected loss. Together with $R_X(\pi,L)\le R_T(\pi,L)$, this proves (i).

Now suppose (ii) fails. Write $P=(P_\theta(x))_{\theta,x}$ and $Q=(Q_\theta(s))_{\theta,s}$, and consider
\[
   \mathcal C
   =\{QK:K\text{ is a Markov kernel from }\Sstate\text{ to }\X\}.
\]
The set $\mathcal C$ is compact and convex in $\mathbb R^{\ThetaSet\times\X}$. Since $P\notin\mathcal C$, strict separation gives a matrix $H=(h_{\theta x})$ satisfying
\[
   \sum_{\theta,x} h_{\theta x}P_\theta(x)
   >
   \sup_K \sum_{\theta,x} h_{\theta x}(QK)_\theta(x).
\]
Take $\A=\X$. Since the sets are finite and $\pi$ has full support, choose $C$ large enough that
\[
    L(\theta,a)=C-\frac{h_{\theta a}}{\pi(\theta)}
\]
is nonnegative for every $(\theta,a)$. With the full observation, the identity rule $a=X$ has Bayes expected loss
\[
    C-\sum_{\theta,x}h_{\theta x}P_\theta(x).
\]
A randomized rule based on $T$ is exactly a Markov kernel from $\Sstate$ to $\X$ and has Bayes expected loss
\[
    C-\sum_{\theta,x}h_{\theta x}(QK)_\theta(x).
\]
The separating inequality therefore makes the loss of the identity rule strictly smaller than the loss of every rule based on $T$. Hence $R_X(\pi,L)<R_T(\pi,L)$, contradicting (i). This establishes (i)$\Rightarrow$(ii). This separation argument is the finite-dimensional form of Blackwell's randomization criterion \parencite{blackwell1951comparison,blackwell1953equivalent}.

We next relate (ii) to ordinary sufficiency. Suppose first that $T$ is sufficient. 
For each $s$ with $\bar Q(s)>0$, sufficiency implies that the conditional distribution $P_\theta(X\in\cdot\mid T=s)$ is the same for every $\theta$ with $Q_\theta(s)>0$. Define $K(\cdot\mid s)$ to be this common conditional distribution, equivalently $\bar P(\cdot\mid T=s)$; for states with $\bar Q(s)=0$, define $K(\cdot\mid s)$ arbitrarily. Then $P_\theta=Q_\theta K$ for every $\theta$, which is (ii).

Conversely, suppose (ii). The statistic map sends $P_\theta$ to $Q_\theta$ and $\bar P$ to $\bar Q$. The data-processing inequality for relative entropy therefore gives
\begin{equation}
   \KL(P_\theta\,\|\,\bar P)
   \ge
   \KL(Q_\theta\,\|\,\bar Q).
   \label{eq:dpi-forward}
\end{equation}
Under (ii), the kernel $K$ sends $Q_\theta$ to $P_\theta$ and $\bar Q$ to $\bar P$, so a second application gives the reverse inequality. Equality holds in \eqref{eq:dpi-forward}. Because $T$ is deterministic, the relative-entropy chain rule yields
\[
 \KL(P_\theta\,\|\,\bar P)
 =
 \KL(Q_\theta\,\|\,\bar Q)
 +
 \sum_{s:\,Q_\theta(s)>0}Q_\theta(s)
 \KL\!\left(P_\theta(\,\cdot\mid T=s)\,\middle\|\,
              \bar P(\,\cdot\mid T=s)\right).
\]
All terms in the sum are nonnegative. Hence, whenever $Q_\theta(s)>0$,
\[
    P_\theta(\,\cdot\mid T=s)
    =
    \bar P(\,\cdot\mid T=s).
\]
The conditional distribution of $X$ given $T$ is independent of $\theta$, so $T$ is sufficient. This proves (iii).
\end{proof}

Theorem~\ref{thm:universal} is a finite specialization of the classical equivalence between decision-theoretic comparison and sufficiency; see \textcite{bahadur1954sufficiency} and \textcite{torgersen1991comparison} for general treatments. Its consequence for compression can be expressed directly through likelihood vectors. Let
\[
    \X_0=\{x\in\X:\bar P(x)>0\},
    \qquad
    \ell(x)=\bigl(P_\theta(x):\theta\in\ThetaSet\bigr),\quad x\in\X_0.
\]

\begin{corollary}[Injectivity under singleton minimal-sufficiency classes]\label{cor:no-compression}
Assume that for all distinct $x,x'\in\X_0$, the likelihood vectors $\ell(x)$ and $\ell(x')$ are not proportional. Every Bayes-risk preserving statistic $T$ is then injective on $\X_0$.
\end{corollary}

\begin{proof}
By Theorem~\ref{thm:universal}, $T$ is sufficient. Fix $s$ with $\bar Q(s)>0$. For every $x\in\X_0$ satisfying $T(x)=s$, sufficiency gives
\[
   P_\theta(x)
   =Q_\theta(s)\bar P(x\mid T=s)
   =\frac{\bar P(x)}{\bar Q(s)}Q_\theta(s)
   \qquad\text{for all }\theta.
\]
If $T(x)=T(x')=s$, then
\[
   P_\theta(x)
   =\frac{\bar P(x)}{\bar P(x')}P_\theta(x')
   \qquad\text{for all }\theta,
\]
so $\ell(x)$ and $\ell(x')$ are proportional. The assumption therefore implies $x=x'$.
\end{proof}

The proportional-likelihood-vector equivalence classes are the usual minimal-sufficiency classes in a finite dominated experiment. Corollary~\ref{cor:no-compression} therefore says that the minimal sufficient statistic is the identity, up to relabelling, when all these classes are singletons \parencite{lehmannscheffe1950completeness,bahadur1954sufficiency,torgersen1991comparison}.

\section{Post-sampling target specification in independent Gaussian streams}

We now consider the acquisition question in a separable model. We use a classical sequential formulation in which the sampling rule and stopping time are part of the procedure. Validity is required for that procedure, not uniformly over other stopping rules. Fix $\Delta>0$ and $\alpha\in(0,1/2)$.

For the one-coordinate benchmark, let the unknown sign be
\[
\sigma \in \{-1,+1\},
\]
and suppose that
\[
Z_n \sim N(\sigma\Delta,1),
\qquad n\geq 1,
\]
independently. Write $\Pp_\sigma$ and $\E_\sigma$ for probability and
expectation when the true sign is $\sigma$.

A randomized sequential test consists of a stopping time $\nu$ and a
terminal decision $\widehat{\sigma}\in\{-1,+1\}$. It is
$\alpha$-correct if
\[
\Pp_{+}(\widehat{\sigma}=-1)\leq \alpha,
\qquad
\Pp_{-}(\widehat{\sigma}=+1)\leq \alpha.
\]
Define
\[
C_1(\alpha,\Delta)
=
\inf
\frac12\left(\E_{+}[\nu]+\E_{-}[\nu]\right),
\]
where the infimum is over all $\alpha$-correct randomized sequential
tests. This is a classical simple-versus-simple sequential testing problem \parencite{wald1947sequential,waldwolfowitz1948optimum}.

For the $m$-coordinate problem, let
\[
      \sigma=(\sigma_1,\ldots,\sigma_m)\in\{-1,+1\}^m.
\]
Write $\Pp_\sigma$ and $\E_\sigma$ for probability and expectation
under configuration $\sigma$.
At calendar time $t$, before the next observation is drawn, a sampling rule chooses $A_t\in\{1,\ldots,m\}$ using the past observations and parameter-independent internal randomization. Conditional on the past and on $A_t$, the next observation satisfies
\[
      Y_t\sim N(\sigma_{A_t}\Delta,1),
\]
and the observations are conditionally independent across times. Let $\tau$ be a stopping time that is almost surely finite under every $\sigma$, and let $S$ be any random element measurable with respect to the stopped history. Write
\[
      N_j(\tau)=\sum_{t=1}^{\tau}\1\{A_t=j\}
\]
for the number of samples from coordinate $j$.

After stopping, an external query may specify a target coordinate $j$. For each $j$ there is a measurable decoder $d_j$ with
\[
      \widehat\sigma_j=d_j(S)\in\{-1,+1\}.
\]
We assume that the query is external to the stopped data $S$. The condition below guarantees validity for each fixed coordinate that might later be requested, but not for a target selected as a function of $S$; the latter is a post-selection or selective-inference problem \parencite{berk2013valid,benjamini2014selective}. We require
\begin{equation}
      \sup_{\sigma\in\{-1,+1\}^m}
      \Pp_\sigma(\widehat\sigma_j\neq\sigma_j)
      \le\alpha,
      \qquad j=1,\ldots,m.
      \label{eq:uniform-correctness}
\end{equation}
Define the symmetric average sample complexity
\[
   C_m^{\mathrm{post}}(\alpha,\Delta)
   =
   \inf
   2^{-m}\sum_{\sigma\in\{-1,+1\}^m}\E_\sigma[\tau],
\]
where the infimum is over all procedures satisfying \eqref{eq:uniform-correctness}.

\begin{theorem}[Exact additive acquisition cost]\label{thm:linear-price}
For every integer $m\ge1$,
\[
      C_m^{\mathrm{post}}(\alpha,\Delta)
      =
      m\,C_1(\alpha,\Delta).
\]
\end{theorem}

\begin{proof}
For the upper bound, fix $\varepsilon>0$ and take an $\alpha$-correct one-coordinate procedure with symmetric average sample size at most $C_1(\alpha,\Delta)+\varepsilon$. Apply an independent copy of this procedure to each coordinate, sequentially in any fixed order, and retain all $m$ terminal decisions. The coordinatewise guarantee \eqref{eq:uniform-correctness} holds. Averaging over the uniform distribution on $\{-1,+1\}^m$ gives expected total sample size at most
\[
       m\bigl(C_1(\alpha,\Delta)+\varepsilon\bigr).
\]
Letting $\varepsilon\downarrow0$ proves
$C_m^{\mathrm{post}}\le mC_1$.

For the reverse inequality, consider any procedure satisfying \eqref{eq:uniform-correctness} and fix $j\in\{1,\ldots,m\}$. We construct a one-coordinate procedure. Draw auxiliary signs
\[
      U_k\in\{-1,+1\},\qquad k\neq j,
\]
independently and uniformly, and generate independent auxiliary Gaussian streams with means $U_k\Delta$ for these coordinates. All auxiliary signs, streams, and randomization used by the emulation are independent of the genuine sign $\sigma$ and may therefore be viewed as parameter-independent internal randomization of the one-coordinate procedure. Emulate the $m$-coordinate procedure. Whenever it requests coordinate $k\neq j$, supply the next observation from the corresponding auxiliary stream. Whenever it requests coordinate $j$, supply the next genuine observation from the one-coordinate stream $N(\sigma\Delta,1)$. When the emulated procedure stops, output $d_j(S)$. If the independent auxiliary randomness is included in the internal-randomization sigma-field, the number of genuine observations used is a randomized stopping time for the genuine stream and the output is measurable at that time.

Conditional on $\sigma$ and $(U_k)_{k\neq j}$, the emulated stopped
history has the same distribution as the original $m$-coordinate
procedure under the parameter vector whose $j$th entry is $\sigma$ and
whose $k$th entry is $U_k$ for every $k\neq j$.Therefore its error probability is at most $\alpha$ for either value of $\sigma$. Its number of genuine observations equals $N_j(\tau)$ in the emulated experiment. Consequently, its symmetric average expected sample size is
\[
  2^{-m}\sum_{\sigma'\in\{-1,+1\}^m}
  \E_{\sigma'}[N_j(\tau)].
\]
By the definition of $C_1(\alpha,\Delta)$,
\[
  2^{-m}\sum_{\sigma'}
  \E_{\sigma'}[N_j(\tau)]
  \ge C_1(\alpha,\Delta).
\]
This inequality holds for every $j$. Since $\tau=\sum_{j=1}^mN_j(\tau)$,
\[
  2^{-m}\sum_{\sigma'}\E_{\sigma'}[\tau]
  \ge m\,C_1(\alpha,\Delta).
\]
Taking the infimum over all $m$-coordinate procedures satisfying
\eqref{eq:uniform-correctness} proves the claim.
\end{proof}

The identity uses only coordinatewise error control, not a simultaneous guarantee. In particular, Theorem~\ref{thm:linear-price} is an exact finite-sample identity for every $(\alpha,\Delta,m)$ under this external-query formulation, rather than an asymptotic rate statement. The factor $m$ comes from the separable information structure: every coordinate must remain answerable, while samples from one coordinate carry no information about another. It is therefore not a general price for an unknown future target. Sequential multiple-testing work treats stronger joint criteria and more general information structures; see, for example, \textcite{songfellouris2019sequential}. The following corollary gives an explicit information lower bound for the Gaussian model.

For $p,q\in(0,1)$, write
\[
   \kl(p,q)=p\log\frac pq+(1-p)\log\frac{1-p}{1-q}.
\]

\begin{corollary}[Information lower bound]\label{cor:kl-bound}
Every $\alpha$-correct one-coordinate procedure satisfies
\[
   \E_+[\nu]\ge
   \frac{\kl(1-\alpha,\alpha)}{2\Delta^2},
   \qquad
   \E_-[\nu]\ge
   \frac{\kl(1-\alpha,\alpha)}{2\Delta^2}.
\]
Hence
\[
   C_m^{\mathrm{post}}(\alpha,\Delta)
   \ge
   \frac{m\,\kl(1-\alpha,\alpha)}{2\Delta^2}.
\]
For every $m$-coordinate procedure satisfying
\eqref{eq:uniform-correctness}, every $\sigma\in\{-1,+1\}^m$, and every $j$,
\[
   \E_\sigma[N_j(\tau)]
   \ge
   \frac{\kl(1-\alpha,\alpha)}{2\Delta^2},
\]
and consequently
\[
   \E_\sigma[\tau]
   \ge
   \frac{m\,\kl(1-\alpha,\alpha)}{2\Delta^2}.
\]
\end{corollary}

\begin{proof}
Consider a one-coordinate procedure and the event
\[
      B=\{\widehat\sigma=+1\}.
\]
Let $\Pp_+^\nu$ and $\Pp_-^\nu$ denote the laws of the complete stopped transcript, including the parameter-independent internal randomization. By data processing through the map that records only $\1_B$,
\[
   \KL(\Pp_+^\nu\,\|\,\Pp_-^\nu)
   \ge
   \kl\bigl(\Pp_+(B),\Pp_-(B)\bigr).
\]
The error constraints imply $\Pp_+(B)\ge1-\alpha$ and $\Pp_-(B)\le\alpha$. Since $\alpha<1/2$ and binary relative entropy is increasing in its first argument and decreasing in its second on the region $p>q$,
\begin{equation}
   \KL(\Pp_+^\nu\,\|\,\Pp_-^\nu)
   \ge
   \kl(1-\alpha,\alpha).
   \label{eq:binary-kl}
\end{equation}
For one observation $Z\sim N(\Delta,1)$,
\[
   \log\frac{dN(\Delta,1)}{dN(-\Delta,1)}(Z)=2\Delta Z,
\]
and its expectation under $N(\Delta,1)$ is $2\Delta^2$. If $\E_+[\nu]<\infty$, the likelihood ratio of the stopped transcript is the product of the observation likelihood ratios up to $\nu$; the internal randomization has the same law under both hypotheses and cancels. Hence
\[
   \KL(\Pp_+^\nu\,\|\,\Pp_-^\nu)
   =
   \E_+\!\left[\sum_{n=1}^{\nu}2\Delta Z_n\right].
\]
Because $\{\nu\ge n\}$ is measurable before $Z_n$ is observed,
\begin{align*}
   \E_+\!\left[\sum_{n=1}^{\nu}2\Delta Z_n\right]
   &=2\Delta\sum_{n\ge1}
      \E_+[\1\{\nu\ge n\}Z_n]\\
   &=2\Delta^2\sum_{n\ge1}\Pp_+(\nu\ge n)\\
   &=2\Delta^2\E_+[\nu].
\end{align*}
Absolute integrability follows from $\E_+[\nu]<\infty$ and the finite first moment of the Gaussian distribution. Combining this identity with \eqref{eq:binary-kl} gives the first lower bound. If $\E_+[\nu]=\infty$, the bound is immediate. Interchanging the two hypotheses gives the bound under $\sigma=-1$.

For the pointwise $m$-coordinate statement, fix $\sigma$ and $j$, and
let $\sigma^{(j)}$ be the vector obtained by reversing only the $j$th
sign. Let $\Pp_\sigma^\tau$ and $\Pp_{\sigma^{(j)}}^\tau$ denote the
laws of the complete stopped transcript under these two configurations,
including the parameter-independent internal randomization. The two stopped experiments differ only in observations obtained when $A_t=j$. The sampling rule and internal randomization are parameter-independent conditional on the observed past, so their factors cancel from the likelihood ratio. The stopped log-likelihood ratio is therefore
\[
  \sum_{t=1}^{\tau}
  \1\{A_t=j\}
  \log
  \frac{dN(\sigma_j\Delta,1)}{dN(-\sigma_j\Delta,1)}(Y_t).
\]
When $A_t=j$, the conditional expected increment under $\sigma$ is $2\Delta^2$. The same predictable-summation argument yields
\[
   \KL(\Pp_\sigma^\tau\,\|\,\Pp_{\sigma^{(j)}}^\tau)
   =2\Delta^2\E_\sigma[N_j(\tau)]
\]
whenever the expectation is finite. This is the standard sequential change-of-measure identity used in fixed-confidence lower bounds; compare, for example, Lemma~1 of \textcite{kaufmann2016complexity}. The event $\{d_j(S)=\sigma_j\}$ has probability at least $1-\alpha$ under $\sigma$ and at most $\alpha$ under $\sigma^{(j)}$. Applying binary data processing as in \eqref{eq:binary-kl} proves the lower bound for $N_j(\tau)$. Summing over $j$ gives the bound for $\E_\sigma[\tau]$.
\end{proof}

\section{Discussion}

The first result gives an extreme answer to the retention question. If one insists on exact preservation of Bayes risk for every finite decision problem, nothing weaker than sufficiency is enough. When the minimal-sufficiency classes are singletons, Corollary~\ref{cor:no-compression} leaves no compression except relabelling. This conclusion depends on requiring exact preservation for every finite decision problem. With a restricted class of losses, or with approximate preservation, coarser reductions may be possible; classical deficiency treats the latter question \parencite{lecam1964sufficiency,torgersen1991comparison}, and recent work studies universally lossless and loss-specific representations \parencite{gyorfi2023lossless,sevetlidis2026bayes}.

The Gaussian example gives a different extreme. Every coordinate may be queried later, but the information is completely separated across coordinates. The exact sampling cost is therefore additive. This does not say that an unknown future target generally costs a factor $m$: the argument uses the fact that no observation can help with more than one target. When observations are informative about several targets, sequential design and controlled sensing no longer have the separable structure of Theorem~\ref{thm:linear-price} \parencite{chernoff1959sequential,nitinawarat2013controlled,kaufmann2016complexity}.

\section*{AI-assisted editing statement}
OpenAI's ChatGPT was used to assist with language editing, LaTeX cross-reference checking, bibliography verification, and a structured audit of the proofs. The author remains fully responsible for the content.

\section*{Funding}
Rianne de Heide's work was supported by NWO Veni grant number VI.Veni.222.018.

\printbibliography

@inproceedings{blackwell1951comparison,
  author    = {Blackwell, David},
  title     = {Comparison of Experiments},
  editor    = {Neyman, Jerzy},
  booktitle = {Proceedings of the Second Berkeley Symposium on Mathematical Statistics and Probability},
  year      = {1951},
  pages     = {93--102},
  publisher = {University of California Press},
  location  = {Berkeley, CA},
  doi       = {10.1525/9780520411586-009},
}

@article{blackwell1953equivalent,
  author       = {Blackwell, David},
  title        = {Equivalent Comparisons of Experiments},
  journaltitle = {The Annals of Mathematical Statistics},
  year         = {1953},
  volume       = {24},
  number       = {2},
  pages        = {265--272},
  doi          = {10.1214/aoms/1177729032},
}

@article{bahadur1954sufficiency,
  author       = {Bahadur, R. R.},
  title        = {Sufficiency and Statistical Decision Functions},
  journaltitle = {The Annals of Mathematical Statistics},
  year         = {1954},
  volume       = {25},
  number       = {3},
  pages        = {423--462},
  doi          = {10.1214/aoms/1177728715},
}

@article{lehmannscheffe1950completeness,
  author       = {Lehmann, E. L. and Scheff{\'e}, Henry},
  title        = {Completeness, Similar Regions, and Unbiased Estimation: Part I},
  journaltitle = {Sankhy\={a}: The Indian Journal of Statistics},
  year         = {1950},
  volume       = {10},
  number       = {4},
  pages        = {305--340},
  url          = {https://www.jstor.org/stable/25048038},
}

@article{lecam1964sufficiency,
  author       = {Le Cam, Lucien},
  title        = {Sufficiency and Approximate Sufficiency},
  journaltitle = {The Annals of Mathematical Statistics},
  year         = {1964},
  volume       = {35},
  number       = {4},
  pages        = {1419--1455},
  doi          = {10.1214/aoms/1177700372},
}

@book{torgersen1991comparison,
  author    = {Torgersen, Erik},
  title     = {Comparison of Statistical Experiments},
  year      = {1991},
  publisher = {Cambridge University Press},
  location  = {Cambridge},
  doi       = {10.1017/CBO9780511666353},
}

@book{wald1947sequential,
  author    = {Wald, Abraham},
  title     = {Sequential Analysis},
  year      = {1947},
  publisher = {John Wiley \& Sons},
  location  = {New York},
}

@article{waldwolfowitz1948optimum,
  author       = {Wald, Abraham and Wolfowitz, Jacob},
  title        = {Optimum Character of the Sequential Probability Ratio Test},
  journaltitle = {The Annals of Mathematical Statistics},
  year         = {1948},
  volume       = {19},
  number       = {3},
  pages        = {326--339},
  doi          = {10.1214/aoms/1177730197},
}

@article{chernoff1959sequential,
  author       = {Chernoff, Herman},
  title        = {Sequential Design of Experiments},
  journaltitle = {The Annals of Mathematical Statistics},
  year         = {1959},
  volume       = {30},
  number       = {3},
  pages        = {755--770},
  doi          = {10.1214/aoms/1177706205},
}

@article{kaufmann2016complexity,
  author       = {Kaufmann, Emilie and Capp{\'e}, Olivier and Garivier, Aur{\'e}lien},
  title        = {On the Complexity of Best-Arm Identification in Multi-Armed Bandit Models},
  journaltitle = {Journal of Machine Learning Research},
  year         = {2016},
  volume       = {17},
  number       = {1},
  pages        = {1--42},
  url          = {https://jmlr.org/papers/v17/kaufman16a.html},
}

@article{songfellouris2019sequential,
  author       = {Song, Yanglei and Fellouris, Georgios},
  title        = {Sequential Multiple Testing with Generalized Error Control: An Asymptotic Optimality Theory},
  journaltitle = {The Annals of Statistics},
  year         = {2019},
  volume       = {47},
  number       = {3},
  pages        = {1776--1803},
  doi          = {10.1214/18-AOS1737},
}

@article{blackwellramamoorthi1982bayes,
  author       = {Blackwell, David and Ramamoorthi, R. V.},
  title        = {A {Bayes} but Not Classically Sufficient Statistic},
  journaltitle = {The Annals of Statistics},
  year         = {1982},
  volume       = {10},
  number       = {3},
  pages        = {1025--1026},
  doi          = {10.1214/aos/1176345895},
}

@online{sevetlidis2026bayes,
  author       = {Sevetlidis, Vasileios},
  title        = {Bayes-Sufficient Representations in Supervised Learning},
  year         = {2026},
  eprint       = {2606.04045},
  eprinttype   = {arXiv},
  eprintclass  = {cs.LG},
  url          = {https://arxiv.org/abs/2606.04045},
}

@article{nitinawarat2013controlled,
  author       = {Nitinawarat, Sirin and Atia, George K. and Veeravalli, Venugopal V.},
  title        = {Controlled Sensing for Multihypothesis Testing},
  journaltitle = {IEEE Transactions on Automatic Control},
  year         = {2013},
  volume       = {58},
  number       = {10},
  pages        = {2451--2464},
  doi          = {10.1109/TAC.2013.2261188},
}

@article{benjamini2014selective,
  author       = {Benjamini, Yoav and Bogomolov, Marina},
  title        = {Selective Inference on Multiple Families of Hypotheses},
  journaltitle = {Journal of the Royal Statistical Society: Series B (Statistical Methodology)},
  year         = {2014},
  volume       = {76},
  number       = {1},
  pages        = {297--318},
  doi          = {10.1111/rssb.12028},
}

@article{gyorfi2023lossless,
  author       = {Gy{\"o}rfi, L{\'a}szl{\'o} and Linder, Tam{\'a}s and Walk, Harro},
  title        = {Lossless Transformations and Excess Risk Bounds in Statistical Inference},
  journaltitle = {Entropy},
  year         = {2023},
  volume       = {25},
  number       = {10},
  eid          = {1394},
  doi          = {10.3390/e25101394},
}

@article{berk2013valid,
  author       = {Berk, Richard and Brown, Lawrence and Buja, Andreas and Zhang, Kai and Zhao, Linda},
  title        = {Valid Post-Selection Inference},
  journaltitle = {The Annals of Statistics},
  year         = {2013},
  volume       = {41},
  number       = {2},
  pages        = {802--837},
  doi          = {10.1214/12-AOS1077},
}
\end{document}